\documentclass[12pt,twoside]{amsart}
\usepackage{latexsym,amsfonts,amsmath,amssymb,color,mathrsfs}

\newtheorem{Th}{Theorem}[section]
\newtheorem{Rem}[Th]{Remark}
\newtheorem{Ex}[Th]{Example}
\newtheorem{Lemma}[Th]{Lemma}
\newtheorem{Def}[Th]{Definition}

\newtheorem{Cor}[Th]{Corollary}

\catcode`\@=11

\renewcommand{\section}%
   {\setcounter{equation}{0}\@startsection {section}{1}{\z@}{-3.5ex plus -1ex
  minus -.2ex}{2.3ex plus .2ex}{\Large\bf}}

\def\grad{\mathop{\rm grad}\nolimits}
\def\ds{\displaystyle}
\def\R{\mathbb R}
\def\C{\mathbb C}
\def\N{\mathbb N}
\def\Q{\mathbb Q}
\def\proj{\mathop{\mbox{\rm proj}}}
\def\ind{\mathop{\mbox{\rm ind}}}

\newcommand{\E}{\mathcal{E}}

\newcommand{\A}{\mathcal{A}}
\newcommand{\K}{\mathcal{K}}

\newcommand{\condH}{\mathscr{H}}
\newcommand{\condC}{\mathscr{C}}

\newcommand{\afrac}[2]{\genfrac{}{}{0pt}{1}{#1}{#2}}

\newcommand{\beqsn}{\arraycolsep1.5pt\begin{eqnarray*}}
\newcommand{\eeqsn}{\end{eqnarray*}\arraycolsep5pt}
\newcommand{\beqs}{\arraycolsep1.5pt\begin{eqnarray}}
\newcommand{\eeqs}{\end{eqnarray}\arraycolsep5pt}

\title{Iterates of systems of operators in spaces of 
$\omega$-ultradifferentiable functions}

\author{C.~Boiti$^*$}
\address{
Dipartimento di Matematica e Informatica \\Universit\`a di Ferrara\\
Via Ma\-chia\-vel\-li n.~30\\
I-44121 Ferrara\\
Italy}
\email{chiara.boiti@unife.it}

\author{R.~Cha\"ili$^{**}$}
\address{
University of sciences and technology of Oran\\ 
Algeria}
\email{rachidchaili@gmail.com; rachid.chaili@univ-usto.dz}

\author{T.~Mahrouz$^{**}$}
\address{
University Ibn Khaldoun of Tiaret\\ 
Algeria}
\email{mahrouz78@gmail.com}

\begin{document}

\begin{abstract}
Given two systems $P=(P_j(D))_{j=1}^N$ and $Q=(Q_j(D))_{j=1}^M$ of linear partial
differential operators with constant coefficients, we consider the spaces
$\E_\omega^P$ and $\E_\omega^Q$ of $\omega$-ultra\-dif\-feren\-tia\-ble 
functions with
respect to the iterates of the systems $P$ and $Q$ respectively.
We find necessary and sufficient conditions, on the systems and on the
weights $\omega(t)$ and $\sigma(t)$, for the inclusion
$\E_\omega^P\subseteq\E_\sigma^Q$.
As a consequence we have a generalization of the classical Theorem of the 
Iterates. 
\end{abstract}

\maketitle
\markboth{\sc Iterates of systems of operators \ldots}
{\sc C.~Boiti, R.~Cha\"ili, T.~Mahrouz}

\vspace{3mm}

\noindent {\em 2010 Mathematics Subject Classification\,}:
{Primary: 35E20; Secondary: 46E10, 35H99.}

\section{Introduction}

The\let\thefootnote\relax\footnote{${}^*$Partially supported by 
the ``INdAM-GNAMPA Project
2015''.\\
\hspace*{3mm}${}^{**}$Supported by ``Laboratoire d'analyse 
math\'ematique et applications, Universit\'e d'Oran 1, Alg\'erie''.} problem of iterates was first introduced by Komatsu \cite{K1} in the
60's, when he characterized analytic functions $u$ on an open subset 
$\Omega\subseteq\R^n$ in terms of the behaviour of successive iterates 
$P^j(D)u$ for a linear partial differential elliptic operator $P(D)$
with constant coefficients.
He proved that, if $P(D)$ is an elliptic operator of order $m$, 
then a $C^\infty$ function $u$ is 
real analytic in $\Omega$ if and only if for every compact 
$K\subset\subset\Omega$ there is a constant $C>0$ such that
\beqs
\label{I1}
\|P^j(D)u\|_{L^2(K)}\leq C^{j+1}(j!)^m,\qquad
\forall j\in\N_0:=\N\cup\{0\}.
\eeqs
This is known as the Theorem of the Iterates.

Moreover, the condition that $P(D)$ is elliptic is sufficient and also
necessary (cf. \cite{M}, \cite{LW}) for the above mentioned result, so that, 
given a linear partial differential operator $P(D)$ of order $m$ with
constant coefficients, the ellipticity growth condition
\beqs
\label{I2}
|\xi|^{2m}\leq C(1+|P(\xi)|^2),\qquad
\forall\xi\in\R^n,
\eeqs
for a constant $C>0$, is equivalent to the equality
\beqsn
\A(\Omega)=\A^P(\Omega),
\eeqsn
where $\A(\Omega)$ is the space of real analytic functions on $\Omega$ and
$\A^P(\Omega)$ is the space of real analytic functions on $\Omega$ with
respect to the iterates of $P$, i.e. the space of $C^\infty$ functions $u$ on
$\Omega$ satisfying \eqref{I1}.

This problem was generalized by Newberger and Zielezny \cite{NZ} to the
class of Gevrey functions proving, more in general, that, for a
pair of hypoelliptic linear partial differential operators $P(D)$ and $Q(D)$ 
with constant coefficients, of order $m$ and $r$ respectively, the condition that
\beqs
\label{I4}
|Q(\xi)|^2\leq C(1+|P(\xi)|^2)^h,
\qquad\forall\xi\in\R^n,
\eeqs
for some $h>0$, is equivalent to an inclusion of the form
\beqsn
\E_{\{t^{\frac1s}\}}^P(\Omega)\subseteq\E_{\{t^{\frac{r}{smh}}\}}^Q(\Omega)
\eeqsn
if $s$ is large enough, where $\E_{\{t^{1/s}\}}^P(\Omega)$ is the space of 
Gevrey functions 
of order $s$ with respect to the iterates of $P=P(D)$, as defined is 
\eqref{EPR} for a
Gevrey weight $\omega(t)=t^{1/s}$.

This result was generalized to the class of $\omega$-ultradifferentiable
functions in the sense of \cite{BMT} by \cite{J}, and was considered in the 
case of systems of operators in 
the Gevrey setting by \cite{BC}.
Here we implement both papers \cite{J} and \cite{BC}, considering the case
of systems in the spaces of $\omega$-ultradifferentiable functions.

In Section \ref{sec1} we define the spaces of $\omega$-ultradifferentiable
functions $\E_\omega^P(\Omega)$ with respect to the iterates of the system
$P=(P_j(D))_{j=1}^N$, both in the Beurling and in the Roumieu setting.

In Section \ref{sec2} we prove that, given two systems $P=(P_j(D))_{j=1}^N$
and $Q=(Q_j(D))_{j=1}^M$ of order $m$ and $r$ respectively, the condition
\beqsn
\sum_{j=1}^M|Q_j(\xi)|\leq C\bigg(1+\sum_{j=1}^N|P_j(\xi)|\bigg)^h,
\qquad\forall\xi\in\R^n,
\eeqsn
is necessary and sufficient for an inclusion of the form
\beqsn
\E^P_{\omega'}(\Omega)\subseteq\E_{\sigma'}^Q(\Omega),
\eeqsn
under assumptions weaker than hypoellipticity (condition $(\condH)$ for
the sufficiency in Theorem~\ref{th1} and condition $(\condC)$ for the
necessity in Theorem~\ref{th2}),
where $\sigma'(t)=\omega'(t^{\frac{r}{mh}})$
with $\omega'(t)=\omega(t^{1/s})$ and $s$ large enough, both in the
Beurling and in the Roumieu setting, for a non-quasianalytic weight 
$\omega$.

In particular, if $P=(P_j(D))_{j=1}^N$ is an elliptic system,
we obtain the Theorem of the Iterates (see Corollary~\ref{cor1}), i.e.
\beqsn
\E^P_{\omega'}(\Omega)=\E_{\omega'}(\Omega).
\eeqsn
Moreover, we prove that the ellipticity of the system $P$ is also
necessary (see Corollary~\ref{cor2}).

In Example~\ref{exA} we have an application of the above results.

Let us finally recall that the Theorem of the Iterates has
also been generalized to the case of variable coefficients, for a
single elliptic operator $P(x,D)$.
It has been proved in the class of real analytic functions by
Kotake and Narasimhan \cite{KN}, in the case of Denjoy-Carleman classes
of Roumieu type by Lions and Magenes \cite{LM} and of Beurling type
with some loss of regularity with respect to the coefficients by 
Oldrich \cite{O}, in the classes of $\omega$-ultradifferentiable functions
of Roumieu type, or of Beurling type but with some loss of regularity with
respect to the coefficients, by Boiti and Jornet \cite{BJ}.

For a microlocal version of the Theorem of the Iterates see, for instance,
\cite{BCM}, \cite{BJJ}, \cite{BJ2}, \cite{BJ3}.
For anisotropic Gevrey classes we refer to \cite{Z}, \cite{BC2}.

\section{Spaces of $\omega$-ultradifferentiable functions with respect
to the iterates of a system of operators}
\label{sec1}

Let us first recall, from \cite{BMT}, the notion of weight functions and 
of spaces of $\omega$-ultra\-dif\-fe\-ren\-tiable functions of Beurling and 
Roumieu 
type:
\begin{Def}
\label{def1}
A {\em non-quasianalytic weight function} is a continuous increasing
function $\omega:\ [0,+\infty)\to[0,+\infty)$ with the following properties:
\begin{itemize}
\item[$(\alpha)$]
$\exists L>0$ s.t. \ $\omega(2t)\leq L(\omega(t)+1)\quad\forall t\geq0$;
\item[$(\beta)$]
$\int_1^{+\infty}\frac{\omega(t)}{t^2}dt<+\infty$;
\item[$(\gamma)$]
$\log t=o(\omega(t))$ as $t\to+\infty$;
\item[$(\delta)$]
$\varphi_\omega(t):=\omega(e^t)$ is convex.
\end{itemize}
For $z\in\C^n$ we write $\omega(z)$ for $\omega(|z|)$, where
$|z|=\sum_{j=1}^n|z_j|$.
We write $\varphi$ for $\varphi_\omega$ when it is clear from the context.
\end{Def}

\begin{Rem}
\begin{em}
Condition $(\beta)$ is the condition of non-quasiananlyticity and it will
ensure the existence of $\omega$-ultradifferentiable functions with 
compact support.

In the Beurling setting, condition $(\gamma)$ may be weakened (cf. \cite{BG}),
by the following:
\beqsn
\hspace*{-5cm}
(\gamma)'
\quad\exists a\in\R,\ b>0:\ \omega(t)\geq a+b\log (1+t),
\quad\forall t\geq0.
\eeqsn
\end{em}
\end{Rem}

The {\em Young conjugate} $\varphi^*$ of $\varphi$ is defined by
\beqsn
\varphi^*(s):=\sup_{t\geq0}\{st-\varphi(t)\},\qquad s\geq0.
\eeqsn
Assuming, without any loss of generality, that $\omega$ vanishes on
$[0,1]$, we have that $\varphi^*$ has only non-negative values, it is convex 
and increasing, $\varphi^*(0)=0$, $\varphi^*(s)/s$ is increasing and
$(\varphi^*)^*=\varphi$ (cf. \cite{BMT}).

An easy computation shows that, for every $a>0$:
\beqs
\label{lemma23}
\sigma(t)=\omega(t^a)\qquad\Rightarrow\qquad \varphi^*_\sigma(s)
=\varphi^*_\omega(s/a).
\eeqs

For a compact set $K\subset\R^n$ and $\lambda>0$ we consider the semi-norm
\beqsn
p_{K,\lambda}(u)=\sup_{\alpha\in\N_0^n}\sup_{x\in K}
|D^\alpha u(x)|e^{-\lambda\varphi^*\left(\frac{|\alpha|}{\lambda}\right)}.
\eeqsn
Then
\beqsn
\E_{\omega,\lambda}(K):=\{u\in C^\infty(K):\ p_{K,\lambda}(u)<+\infty\}
\eeqsn
is a Banach space endowed with the norm $p_{K,\lambda}$.

Let us then recall from \cite{BMT} the definition of the space of
{\em $\omega$-ultradifferentiable functions of Beurling type} in an open
set $\Omega\subseteq\R^n$:
\beqsn
\E_{(\omega)}(\Omega):=\proj_{\afrac{\longleftarrow}{K\subset\subset\Omega}}
\proj_{\afrac{\longleftarrow}{\lambda>0}}
\E_{\omega,\lambda}(K).
\eeqsn
This is a Fr\'echet space.

The space of {\em $\omega$-ultradifferentiable functions of Roumieu type}
is defined by
\beqsn
\E_{\{\omega\}}(\Omega):=\proj_{\afrac{\longleftarrow}{K\subset\subset\Omega}}
\ind_{\afrac{\longrightarrow}{m\in\N}}\E_{\omega,\frac1m}(K).
\eeqsn

Let us now consider a system $P=(P_j(D))_{j=1}^N$ of linear partial differential
operators with constant coefficients. For $\beta\in\N_0^N$ we define the 
iterates of the system $P$ as
\beqsn
P^\beta:=P_1^{\beta_1}(D)\circ P_2^{\beta_2}(D)\circ\cdots\circ P_N^{\beta_N}(D),
\eeqsn
where $P_j^{\beta_j}(D)$ is the $\beta_j$-th iterate of the operator $P_j(D)$,
i.e. 
\beqsn
P_j^{\beta_j}(D)=\underbrace{P_j(D)\circ\cdots\circ P_j(D)}_{\beta_j},
\eeqsn
and $P^0(D)u=u$.

We shall say, in the following, that the system $P=(P_j(D))_{j=1}^N$ 
has order $m$
if each operator $P_j(D)$ has order $m$. In this case, for a compact
$K\subset\R^n$ and $\lambda>0$ we consider the semi-norm
\beqsn
p_{K,\lambda}^P(u):=\sup_{\beta\in\N_0^N}\|P^\beta u\|_{L^2(K)}
e^{-\lambda\varphi^*\left(\frac{|\beta|m}{\lambda}\right)}
\eeqsn
and define
\beqs
\label{1}
\E^P_{\omega,\lambda}(K):=\{u\in C^\infty(K):\ p^P_{K,\lambda}(u)<+\infty\}.
\eeqs

For an open set $\Omega\subseteq\R^n$ we define the space of
{\em $\omega$-ultradifferentiable functions of Beurling type with respect
to the iterates of the system} $P=(P_j(D))_{j=1}^N$ by:
\beqs
\label{EPB}
\E^P_{(\omega)}(\Omega):=\proj_{\afrac{\longleftarrow}{K\subset\subset\Omega}}
\proj_{\afrac{\longleftarrow}{\lambda>0}}
\E^P_{\omega,\lambda}(K).
\eeqs

Analogously, we define the
space of {\em $\omega$-ultradifferentiable functions of Roumieu type 
with respect to the iterates of the system} $P$ by:
\beqs
\label{EPR}
\E^P_{\{\omega\}}(\Omega):=\proj_{\afrac{\longleftarrow}{K\subset\subset\Omega}}
\ind_{\afrac{\longrightarrow}{\ell\in\N}}\E^P_{\omega,\frac1\ell}(K).
\eeqs

\vspace{1mm}
\noindent
{\bf Notation.} In the following we shall write $\E_\omega^P(\Omega)$ if the
statement holds both in the Beurling case $\E_{(\omega)}^P(\Omega)$ and in the
Roumieu case $\E_{\{\omega\}}^P(\Omega)$.

\begin{Rem}
\begin{em}
  When the system is given by a single operator $P=P(D)$, the above defined
  spaces $\E_\omega^P(\Omega)$ coincide with the corresponding ones defined
  in \cite{BJJ} (see \cite{J} for the original, slightly different,
  definition).
  \end{em}
\end{Rem}

Analogously as in \cite{BC}, we give the following:
\begin{Def}
  \label{defC}
  We say that the system $P=(P_j(D))_{j=1}^N$ satisfies condition $(\condC)$ if
  for every $\lambda>0$ and $K\subset\subset\Omega$ the space
  $\E_{\omega,\lambda}^P(K)$ defined in \eqref{1} is a Banach space endowed with 
the norm $p_{K,\lambda}^P$.
    \end{Def}

Let $\{K_\ell\}_{\ell\in\N}$ be a compact exhaustion of $\Omega$, i.e. a
 sequence of compact subsets of $\Omega$ with
$K\subset{\buildrel\circ\over{K}}_{\ell+1}$ and $\cup_\ell K_\ell=\Omega$. We 
have that
\beqs
\label{2F}
\E_{(\omega)}^P(\Omega)=\proj_{\afrac{\longleftarrow}{\ell\in\N}}
\proj_{\afrac{\longleftarrow}{m\in\N}}\E_{\omega,m}^P(K_\ell)
=\proj_{\afrac{\longleftarrow}{\ell\in\N}}\E_{\omega,\ell}^P(K_\ell).
\eeqs

\begin{Rem}
\label{rem210bis}
\begin{em}
If condition $(\condC)$ is satisfied, then $\E_{(\omega)}^P(\Omega)$, endowed
with the metrizable local convex topology defined by the fundamental
system of semi-norms $\{p^P_{K_\ell,\ell}\}_{\ell\in\N}$, is a Fr\'echet space.
On the contrary, condition $(\condC)$ does not garantee that 
$E^P_{\{\omega\}}(\Omega)$
is complete.

However, if $P=(P_j(D))_{j=1}^N$ is a system of hypoelliptic operators, then 
it can be proved,
as in \cite[Thm. 3.3]{J}, that both $\E^P_{(\omega)}(\Omega)$ and
$\E^P_{\{\omega\}}(\Omega)$ are complete.

In the case of a single operator $P=P(D)$ it was 
proved in \cite[Prop. 3.1]{J} that
also the converse is valid: if $\E_{\omega}^P(\Omega)$ is complete,
then $P(D)$ must be hypoelliptic.
This is not true in the case of systems. Take, for instance,
$P=(D_j)_{j=1}^n$ for $D_j=-i\partial_{x_j}$. Then 
$\E_{\omega}^P(\Omega)=\E_{\omega}(\Omega)$ is complete by 
\cite[Prop. 4.9]{BMT}, but the operators $P_j(D)=D_j$ are not hypoelliptic.
\end{em}
\end{Rem}

\begin{Rem}
\label{rem210tris}
\begin{em}
It is possible to construct a finer locally convex topology that makes
$\E_\omega^P(\Omega)$ always complete, without any assumption on the operators.

In the Beurling case we take a compact exhaustion $\{K_\ell\}_{\ell\in\N}$ of
$\Omega$, set
\beqsn
p_\ell(u):=\sup_{|\alpha|\leq \ell}\sup_{x\in K_\ell}|D^\alpha u(x)|
\eeqsn
and then consider the semi-norm
\beqsn
\tau_\ell^P(u):=\max\left\{p_{K_\ell,\ell}^P(u),p_\ell(u)\right\}.
\eeqsn

We have that $\E_{(\omega)}^P(\Omega)$, endowed with the convex topology
defined by the fundamental system of semi-norms $\{\tau^P_\ell\}_{\ell\in\N}$, is a
Fr\'echet space. The proof is standard.

In the Roumieu case we consider, for
$\ell\in\N$ and $K\subset\subset\Omega$, the fundamental system of
semi-norms $\{\tau_{K,\ell,m}^P\}_{m\in\N}$ defined by
\beqs
\label{defsemi}
\tau_{K,\ell,m}^P(u):=\max\{p_{K,\frac1\ell}^P(u),p_m(u)\}.
\eeqs
This makes $\E_{\omega,\frac1\ell}^P(K)$ a Fr\'echet space.
Considering then on $\E_{\{\omega\}}^P(\Omega)$ the topology induced by
\eqref{EPR}, we can prove, as in \cite[Prop. 3.5]{J}, that 
$\E_{\{\omega\}}^P(\Omega)$ is complete.
\end{em}
\end{Rem}

  We now want to look for sufficient and necessary conditions in order to
  obtain the Theorem of the Iterates for systems $P=(P_j(D))_{j=1}^N$
  of linear partial differential operators with constant coefficients
  in the classes of $\omega$-ultradifferentiable functions.
  
\section{A sufficient condition}
\label{sec2}

Analogously as in \cite{BC}, we give the following:
\begin{Def}
\label{defH}
Let $P=(P_j(D))_{j=1}^N$ be a system of linear partial differential 
operators with constant coefficients of order $m$. We say that $P$ satisfies
{\em condition} $(\condH)$ if there exist $C>0$ and $\gamma\geq m$
such that
\beqs
\label{eqH}
\sum_{j=1}^N|P^{(\alpha)}_j(\xi)|\leq C\bigg(1+\sum_{j=1}^N|P_j(\xi)|
\bigg)^{1-\frac{|\alpha|}{\gamma}},\qquad\forall\alpha\in\N_0^n,\xi\in\R^n,
\eeqs
where $P_j^{(\alpha)}(\xi)=\partial_\xi^\alpha P_j(\xi)$.
\end{Def}

\begin{Rem}
\label{remlemma32}
\begin{em}
If the system $P=(P_j(D))_{j=1}^N$ satisfies condition $(\condH)$ for
some $\gamma\geq m$, there exists a smallest $\gamma_P\geq m$ such that
$P$ satisfies \eqref{eqH} for $\gamma=\gamma_P$; moreover $\gamma_P\in\Q$.
Indeed, the inequality \eqref{eqH} implies that there exists $C'>0$ such that
\beqs
\label{eqgrad}
|\grad P_i(\xi)|\leq C'\bigg(1+\sum_{j=1}^N|P_j(\xi)|\bigg)^{1-\frac1\gamma},
\qquad\forall i=1,\ldots,N.
\eeqs
Applying then the Tarski-Seidenberg theorem to the semi-algebraic function
\beqsn
M_i(\lambda)=\sup_{\sum_{j=1}^N|P_j(\xi)|=\lambda}|\grad P_i(\xi)|,
\eeqsn
we can argue as in \cite[Thm. 3.1]{H} to prove 
that for every $i\in\{1,\ldots,N\}$ there exists a smallest $\gamma_i$
such that
\beqs
\label{9}
|P_i^{(\alpha)}(\xi)|\leq C\bigg(1+\sum_{j=1}^N
|P_j(\xi)|\bigg)^{1-\frac{|\alpha|}{\gamma_i}},
\qquad\forall\alpha\in\N_0^n,\xi\in\R^n.
\eeqs
Then $\gamma_P:=\max\{\gamma_1,\ldots,\gamma_N\}$ is the smallest
$\gamma$ satisfying \eqref{eqH} and moreover 
$\gamma_P\in\Q$ and $\gamma_P\geq m$.
\end{em}
\end{Rem}

In the following, for a system $P$ satisfying condition $(\condH)$, 
we shall always
refer to $\gamma_P$ as defined in Remark~\ref{remlemma32}.

\begin{Rem}
\label{rem1}
\begin{em}
If $P=P(D)$ is a hypoelliptic operator, then
conditon $(\condH)$ is satisfied because of \cite[Thm. 3.1]{H}.
However, in general condition $(\condH)$ is weaker than hypoellipticity.
Take for instance in $\R^2$ the operator $P(D)=P(D_1,D_2)=D_1^2$. It is 
trivially not hypoelliptic,
but it satisfies condition $(\condH)$ for $\gamma=2$.

More in general, if $P=(P_j(D))_{j=1}^N$ is a system of hypoelliptic
operators, then $P$ satisfies condition $(\condH)$.
If the system $P$ is elliptic, i.e.
\beqs
\label{ellittico}
|\xi|^m\leq C\bigg(1+\sum_{j=1}^N|P_j(\xi)|\bigg),\qquad\forall\xi\in\R^n,
\eeqs
then condition $(\condH)$ is
satisfied for $\gamma_P=m$.
\end{em}
\end{Rem}

In order to compare, for two given systems $P=(P_j(D))_{j=1}^N$
and $Q=(Q_j(D))_{j=1}^M$, the corresponding spaces $\E_\omega^P(\Omega)$ and
$\E_\omega^Q(\Omega)$, we introduce the following:
\begin{Def}
\label{defhweaker}
Let $P=(P_j(D))_{j=1}^N$ and $Q=(Q_j(D))_{j=1}^M$ be two systems of linear 
partial differential operators with constant coefficients. If there exist
$C,h>0$ such that
\beqs
\label{eqh}
\sum_{j=1}^M|Q_j(\xi)|\leq C\bigg(1+\sum_{j=1}^N|P_j(\xi)|\bigg)^h,
\qquad\forall\xi\in\R^n,
\eeqs
we say that $Q$ is {\em $h$-weaker} than $P$ and we write 
$Q\prec_h P$.
\end{Def}

\begin{Rem}
\label{remh}
\begin{em}
If $P=P(D)$ and $Q=Q(D)$ are single operators and $P(D)$ is hypoelliptic, 
then by \cite[Thm. 3.2]{H} there is a smallest $h$ such that $Q$ is 
$h$-weaker than $P$, and
moreover $h\in\Q$.

More in general, if $P=(P_j(D))_{j=1}^N$ is $h$-weaker than $Q=(Q_j(D))_{j=1}^M$,
there exists a smallest $h>0$ such that \eqref{eqh} is satisfied and moreover 
$h\in\Q$.
Indeed, we can argue as in \cite[Thm. 3.2]{H} and Remark~\ref{remlemma32},
 taking the semi-algebraic functions
\beqsn
M_i(\lambda)=\sup_{\sum_{j=1}^N|P_j(\xi)|=\lambda}|Q_i(\xi)|.
\eeqsn
\end{em}
\end{Rem}

\begin{Def}
If $P=(P_j(D))_{j=1}^N$ and $Q=(Q_j(D))_{j=1}^M$ are two systems with
$P\prec_h Q$ and $Q\prec_h P$,  we say that the systems $P$ and $Q$ are 
{\em $h$-equally strong},
and we write $P\approx_h Q$.
\end{Def}

\begin{Rem}
\label{remlemma38}
\begin{em}
Arguing as in \cite[pg 210]{H}, we can easily prove that if
$P=(P_j(D))_{j=1}^N$ and $Q=(Q_j(D))_{j=1}^M$ are two systems of order $m$
and $r$ respectively, satisfying condition $(\condH)$ and 1-equally strong, then $m=r$ 
and $\gamma_P=\gamma_Q$.
\end{em}
\end{Rem}

We are now ready to prove the following result:
\begin{Th}
\label{th1}
Let $P=(P_j(D))_{j=1}^N$ and $Q=(Q_j(D))_{j=1}^M$ be two systems of linear 
partial differential operators with constant coefficients, of order $m$ and 
$r$ respectively.
Assume that $P$ and $Q$ satisfy condition $(\condH)$ of Definition~\ref{defH} and that 
$Q$ is $h$-weaker than $P$. 
Let $\Omega$ be an open subset of $\R^n$.
Let $\omega$ be a non-quasianalytic weight function and set 
$\omega'(t)=\omega(t^{1/s})$ for $s\geq\gamma_P/m$.
Then
\beqs
\label{B1}
\E_{(\omega')}^P(\Omega)\subseteq&&\E_{(\sigma')}^Q(\Omega)\\
\label{R1}
\E_{\{\omega'\}}^P(\Omega)\subseteq&&\E_{\{\sigma'\}}^Q(\Omega),
\eeqs
for $\sigma'(t)=\omega'(t^{\frac{r}{mh}})=\omega(t^{\frac{r}{smh}})$.
\end{Th}

\begin{proof}
\underline{Beurling case:}

Let $u\in\E_{(\omega')}^P(\Omega)$. For every compact $K\subset\Omega$
there exist an open set $F$ relatively compact in $\Omega$ and $\delta>0$
such that
\beqsn
K\subset F_{(M+1)\delta}\subset F\subset\Omega,
\eeqsn
where
\beqsn
F_\sigma:=\{x\in F:\ d(x,\partial F)>\sigma\}.
\eeqsn

Moreover, for every $q\in\N$ there exists $C_q>0$ such that
\beqs
\label{15}
\sum_{|\beta|=\ell}\|P^\beta u\|_{L^2(F)}\leq C_q
e^{q\varphi^*_{\omega'}\left(\frac{\ell m}{q}\right)}
=C_qe^{q\varphi^*_{\omega}\left(\frac{\ell ms}{q}\right)},
\qquad\forall\ell\in\N,
\eeqs
by the definition of $\E_{(\omega')}^P(\Omega)$ and by \eqref{lemma23}.

By assumption $Q\prec_h P$ and, by Remark~\ref{remh}, 
there exists $\mu,\nu\in\N$
such that $h=\mu/\nu$. 

Arguing as in \cite[Thm. 2.4]{BC}, we fix $\alpha\in\N_0^M$,
choose $k_j,\ell_j\in\N_0$ such that $\alpha_j=k_j\nu+\ell_j$, with $l_j\leq\nu-1$, for 
$1\leq j\leq M$, and set $k=\sum_{j=1}^M k_j$.
From \cite[formula (2.12)]{BC} there exist $C_1,C_2>0$ such that for every
$u\in C^\infty(F)$:
\beqs
\nonumber
\|Q^\alpha u\|_{L^2(F_{(M+1)\delta})}\leq C_1^M\bigg[&&\sum_{i=0}^M
\binom Mi M^i C_2^{k+i}\sum_{|\beta|\leq k+i}\binom{k+i}{|\beta|}\\
\label{212}
&&\cdot\left(\frac{k+i}{\delta}\right)^{(k+i-|\beta|)\gamma_P\mu}
\|P^{\beta\mu}u\|_{L^2(F)}\bigg].
\eeqs

If $\gamma_P\leq sm$, then from \eqref{15} we have, for all $\ell\leq k$:
\beqs
\nonumber
&&k^{(k-\ell)\gamma_P\mu}\sum_{|\beta|=\ell}\|P^{\beta\mu}u\|_{L^2(F)}
\leq C_qk^{(k-\ell)sm\mu}e^{q\varphi^*_{\omega}\left(\frac{m\ell \mu s}{q}\right)}\\
\nonumber
&&\leq C_q\left(1+\frac{\ell}{k-\ell}\right)^{\frac{k-\ell}{\ell}sm\mu\ell}
(k-\ell)^{(k-\ell)sm\mu}e^{q\varphi^*_{\omega}\left(\frac{m\ell \mu s}{q}\right)}\\
\label{16}
&&\leq C_qe^{sm\mu\ell}[(k-\ell)sm\mu]^{(k-\ell)sm\mu}
e^{q\varphi^*_{\omega}\left(\frac{m\ell \mu s}{q}\right)}.
\eeqs

Since $\omega(t)$ is a non-quasianalytic weight function, condition
$(\beta)$ implies $\omega(t)=o(t)$ and hence for every $q'\in\N$ there 
exists $C_{q'}>0$ such that from \cite[Rem. 2.4]{AJO}:
\beqs
y\log y\leq y+q'\varphi^*_\omega\left(\frac{y}{q'}\right)+C_{q'},
\qquad\forall y>0.
\eeqs

Applying the above inequality to \eqref{16} we have that
\beqs
\label{17}
k^{(k-\ell)\gamma_P\mu}\sum_{|\beta|=\ell}\|P^{\beta\mu}u\|_{L^2(F)}
\leq C_q e^{sm\mu\ell}e^{(k-\ell)sm\mu}
e^{q'\varphi^*_{\omega}\left(\frac{(k-\ell)m \mu s}{q'}\right)}
e^{C_{q'}}e^{q\varphi^*_{\omega}\left(\frac{m\ell \mu s}{q}\right)}.
\eeqs

By condition $(\alpha)$ of Definition~\ref{def1} there exists $\tilde{L}>0$
such that
\beqsn
\omega(et)\leq\tilde{L}(1+\omega(t)),\qquad\forall t\geq0.
\eeqsn
Then, from \cite[Prop. 21(e) and Rem. 22]{BJ} we have that for every 
$\rho,\lambda>0$ there exists $\lambda',D_{\rho,\lambda}>0$ such that
\beqs
\label{la5}
\rho^je^{\lambda\varphi^*_\omega(j/\lambda)}\leq D_{\rho,\lambda}
e^{\lambda'\varphi^*_\omega(j/\lambda')},\qquad\forall j\in\N_0,
\eeqs
with $\lambda'=\lambda/\tilde{L}^{[\log\rho+1]}$ and
$D_{\rho,\lambda}=\exp\{\lambda[\log\rho+1]\}$, where $[\log\rho+1]$ is the
integer part of $(\log\rho+1)$.

Applying \eqref{la5} in \eqref{17} we have that 
for every $\lambda>0$ there exists $C_\lambda>0$ such that
\beqs
\label{18}
k^{(k-\ell)\gamma_P\mu}\sum_{|\beta|=\ell}\|P^{\beta\mu}u\|_{L^2(F)}
\leq C_{\lambda}e^{\lambda\varphi^*_{\omega}\left(\frac{m\ell \mu s}{\lambda}\right)}
e^{\lambda \varphi^*_{\omega}\left(\frac{(k-\ell)m \mu s}{\lambda}\right)}.
\eeqs

From condition $(\alpha)$ of Definition~\ref{def1}, by \cite[Lemma 1.2]{BMT}
we have that there exists $L'>0$ such that
\beqsn
\omega(u+v)\leq L'(\omega(u)+\omega(v)+1),\qquad
\forall u,v\geq0,
\eeqsn
and hence for all $j,k\in\N_0,\lambda>0$:
\beqs
\nonumber
e^{\lambda\varphi^*\left(\frac j\lambda\right)+
\lambda\varphi^*\left(\frac k\lambda\right)}=&&
\sup_{s\geq0}e^{js-\lambda\varphi_\omega(s)}\cdot
\sup_{t\geq0}e^{kt-\lambda\varphi_\omega(t)}
=\sup_{u,v\geq1}e^{j\log u+k\log v-\lambda(\omega(u)+\omega(v))}\\
\nonumber
\leq&&\sup_{u,v\geq1}u^jv^ke^{-\frac{\lambda}{L'}\omega(u+v)}e^\lambda
\leq e^\lambda\sup_{u,v\geq1}(u+v)^{j+k}e^{-\frac{\lambda}{L'}\omega(u+v)}\\
\label{19}
\leq&&e^\lambda\sup_{\sigma\geq0}e^{(j+k)\sigma-\frac{\lambda}{L'}\varphi_\omega(\sigma)}
=e^\lambda e^{\frac{\lambda}{L'}\varphi^*_\omega\left(\frac{j+k}{\lambda/L'}\right)}.
\eeqs

Applying it to \eqref{18} we have that for every $\tilde{q}\in\N$ there exists
$C_{\tilde{q}}>0$ such that for all $\ell\leq k$:
\beqs
\label{20}
k^{(k-\ell)\gamma_P\mu}\sum_{|\beta|=\ell}\|P^{\beta\mu}u\|_{L^2(F)}
\leq C_{\tilde{q}}e^{\tilde{q}\varphi^*_\omega\left(\frac{k\mu ms}{\tilde{q}}\right)}.
\eeqs

Substituting in \eqref{212} we obtain, for some constant $A>0$:
\beqs
\nonumber
\|Q^\alpha u\|_{L^2(F_{(M+1)\delta})}\leq&&A^kC_{\tilde{q}}
e^{\tilde{q}\varphi^*_\omega\left(\frac{(k+M)\mu ms}{\tilde{q}}\right)}
\leq A^kC_{\tilde{q}}
e^{\frac{\tilde{q}}{2}\varphi^*_\omega\left(\frac{2k\mu ms}{\tilde{q}}\right)}
e^{\frac{\tilde{q}}{2}\varphi^*_\omega\left(\frac{2M\mu ms}{\tilde{q}}\right)}\\
\label{21}
\leq&&C'_{\tilde{q}}A^{\mu msk}
e^{\frac{\tilde{q}}{2}\varphi^*_\omega\left(\frac{k\mu ms}{\tilde{q}/2}\right)}
\leq D_q
e^{q\varphi^*_\omega\left(\frac{k\mu ms}{q}\right)}
\eeqs
by the convexity of $\varphi^*_\omega$ and by \eqref{la5}, for
$q=\tilde{q}/(2\tilde{L}^{[\log A+1]})$.

Since $k\leq|\alpha|/\nu$ by construction, from \eqref{21} we thus have that
for every $q\in\N$ there exists $D_q>0$ such that
\beqsn
\|Q^\alpha u\|_{L^2(K)}\leq\|Q^\alpha u\|_{L^2(F_{(M+1)\delta})}
\leq D_q
e^{q\varphi^*_\omega\left(\frac{|\alpha|\mu ms}{\nu q}\right)}
=D_qe^{q\varphi^*_{\sigma'}\left(\frac{|\alpha|r}{q}\right)},
\qquad\forall\alpha\in\N_0^M,
\eeqsn
by \eqref{lemma23}, since $\sigma'(t)=\omega(t^{\frac{r}{smh}})$.
This proves that $u\in\E^Q_{(\sigma')}(\Omega)$.

\underline{Roumieu case:}

It is similar to the Beurling case: in \eqref{15}  we take 
$\frac 1q\varphi^*_{\omega'}(\ell mq)$
instead of $q\varphi^*_{\omega'}\left({\ell m}/{q}\right)$ and
a fixed constant $C$ instead of $C_q$, and similarly later on for $q',q'',\ldots$.

The proof is complete.
\end{proof}

\begin{Cor}
\label{cor0}
Let $P=(P_j(D))_{j=1}^N$ and $Q=(Q_j(D))_{j=1}^M$ be two systems of order $m$
satisfying condition $(\condH)$ and 1-equally strong.
Let $\Omega$ be an open subset of $\R^n$.
Let $\omega$ be a non-quasianalytic weight function and set
$\omega'(t)=\omega(t^{1/s})$ for $s\geq\gamma_P/m=\gamma_Q/m$.
Then
\beqsn
\E^P_{(\omega')}(\Omega)=\E^Q_{(\omega')}(\Omega)\qquad\mbox{and}\qquad
\E^P_{\{\omega'\}}(\Omega)=\E^Q_{\{\omega'\}}(\Omega).
\eeqsn
\end{Cor}

From Remark~\ref{rem1} we obtain the Theorem of the Iterates as a
corollary of
Theorem~\ref{th1}:
\begin{Cor}
\label{cor1}
Let $P=(P_j(D))_{j=1}^N$ be an elliptic system of order $m$.
Let $\Omega$ be an open subset of $\R^n$ and
$\omega$ a non-quasianalytic weight function.
Then
\beqs
\label{desiredeq}
\E^P_{(\omega)}(\Omega)=\E_{(\omega)}(\Omega)\qquad\mbox{and}\qquad
\E^P_{\{\omega\}}(\Omega)=\E_{\{\omega\}}(\Omega).
\eeqs
\end{Cor}

\begin{proof}
\underline{Beurling case:}

Let us first prove the inclusion
\beqs
\label{dadim}
\E^P_{(\omega)}(\Omega)\subseteq\E_{(\omega)}(\Omega).
\eeqs
To this aim we consider the system $Q=(D_j)_{j=1}^n$, for $D_j=-i\partial_{x_j}$.
The operators $Q_j(D)=D_j$ are not hypoelliptic, but the system $Q$
satisfies conditon $(\condH)$.
The system $P$ satisfies \eqref{ellittico} and hence condition $(\condH)$
for $\gamma_P=m$, by 
Remark~\ref{rem1}.

Since \eqref{ellittico} implies that $Q$ is $\frac 1m$-weaker than $P$,
from Theorem~\ref{th1}, with $s=1=\gamma_P/m$ and hence $\omega'(t)=\omega(t)$,
we have that
\beqs
\label{30}
\E^P_{(\omega)}(\Omega)\subseteq\E^Q_{(\sigma)}(\Omega)=
\E_{(\sigma)}(\Omega),
\eeqs
for $\sigma(t)=\omega(t^{{1}/(m\cdot\frac 1m)})=\omega(t)$, and 
hence \eqref{dadim} is proved.

Vice versa, since every $P_j(\xi)$ is a polynomial of degree $m$, we clearly
have that $P$ is $m$-weaker than $Q$ and, from Theorem~\ref{th1},
\beqsn
\E_{(\omega)}(\Omega)=\E^Q_{(\omega)}(\Omega)\subseteq\E^P_{(\sigma)}(\Omega)
\eeqsn
for $\sigma(t)=\omega\left(t^{\frac{m}{1\cdot m}}\right)=\omega(t)$, so that
also the opposite inclusion
\beqsn
\E_{(\omega)}(\Omega)\subseteq\E^P_{(\omega)}(\Omega)
\eeqsn
is valid, and hence the equality \eqref{desiredeq} is proved in the 
Beurling case.

\underline{Roumieu case:}

The proof is the same as in the Beurling case, using \eqref{R1} instead of
\eqref{B1}.
\end{proof}

\begin{Ex}
\label{exA}
\begin{em}
Let us consider in $\R^2$ the system $P=(P_j(D))_{j=1}^2$ defined by
\beqsn
P_1(D_1,D_2)=D_1^2,\qquad\qquad P_2(D_1,D_2)=D_2^2.
\eeqsn
These operators are not hypoelliptic but the system $P$ satisfies conditon
$(\condH)$ for $\gamma_P=2$.

Let us then condider $Q=Q(D)=\Delta=-D_1^2-D_2^2$.
This is an elliptic operator of order 2 and hence satisfies condition
$(\condH)$ for $\gamma_Q=2$ (see Remark~\ref{rem1}).

Moreover, $P$ and $Q$ are 1-equally strong and $\gamma_P/m=1$. We can then
apply Corollaries~\ref{cor0} and \ref{cor1} with
$\omega'(t)=\omega(t)$ and obtain that, for any open subset $\Omega$ of $\R^2$
and for every non-quasianalytic weight function $\omega$:
\beqsn
\E^P_{(\omega)}(\Omega)=\E^Q_{(\omega)}(\Omega)=\E_{(\omega)}(\Omega).
\eeqsn

This means that the elements $u\in\E_{(\omega)}(\Omega)$ can be equivalently
determined by estimating their derivatives
$D^\alpha u(x)=D_1^{\alpha_1}D_2^{\alpha_2}u(x)$, or the iterates of
$Q(D)$, i.e. $\Delta^\beta u(x)$, or the iterates of the system
$P=(P_j(D))_{j=1}^2$, i.e.
$P^\gamma u(x)=D_1^{2\gamma_1}D_2^{2\gamma_2}u(x)$,
for $\alpha,\gamma\in\N_0^2$, $\beta\in\N_0$.

The same holds also in the Roumieu case.
\end{em}
\end{Ex}

\section{A necessary condition}
\label{sec3}

In order to obtain a necessary condition for the inclusions
\eqref{B1} or \eqref{R1}, we first need to introduce the following:
\begin{Def}
We say that a non-quasianalytic weight function $\omega$ satisfies the
growth {\em condition B-M-M} if there exists a constant $H\geq1$ such that
\beqs
\label{BMM}
2\omega(t)\leq\omega(Ht)+H,\qquad\forall t\geq0.
\eeqs
\end{Def}

\begin{Rem}
\begin{em}
Condition B-M-M was introduced in \cite{BMM} in order to characterize 
those weight functions $\omega$ for which $\E_{(\omega)}(\Omega)$
(or $\E_{\{\omega\}}(\Omega)$) can be also considered as a Denjoy-Carleman
class $\E_{(M_p)}(\Omega)$ (or $\E_{\{M_p\}}(\Omega)$, respectively) as defined
in \cite{K}, for some sequence $\{M_p\}$.

Gevrey weights satisfy condition B-M-M.
\end{em}
\end{Rem}

Let us now prove that the condition $Q\prec_h P$ of Theorem~\ref{th1}
is also necessary for the inclusions \eqref{B1} and \eqref{R1}.

To this aim we first recall, from \cite[Lemma 4.7]{J}, the following:
\begin{Lemma}
\label{lemmaJ}
For all $h,\lambda>0$ and $t\geq1$ we have that:
\begin{itemize}
\item[(i)]
$\ds\sup_{j\in\N_0}t^j
\exp\left\{-\lambda\varphi^*\left(\frac{hj}{\lambda}\right)\right\}\leq
\exp\left\{\lambda\omega(t^{1/h})\right\}$;
\item[(ii)]
$\ds\sup_{j\in\N_0}t^j
\exp\left\{-\lambda\varphi^*\left(\frac{hj}{\lambda}\right)\right\}\geq
\frac1t\exp\left\{\lambda\omega(t^{1/h})\right\}$.
\end{itemize}
\end{Lemma}

We can then prove:
\begin{Th}
\label{th2}
Let $\Omega$ be an open subset of $\R^n$ and $\omega$ a 
non-quasianalytic weight function satisfying condition B-M-M. Let
$P=(P_j(D))_{j=1}^N$ be a system of linear partial differential operators
of order $m$ with constant coefficients satisfying condition
$(\condC)$ of Definition~\ref{defC} and let $Q=(Q_j(D))_{j=1}^M$ 
be a generic system of linear
partial differential operators
of order $r$ with constant coefficients.

If there exists $h>0$ such that one of the following inclusions
\beqs
\label{B2}
\E^P_{(\omega)}(\Omega)\subseteq\E^Q_{(\sigma)}(\Omega)
\eeqs
or
\beqs
\label{R2}
\E^P_{\{\omega\}}(\Omega)\subseteq\E^Q_{\{\sigma\}}(\Omega)
\eeqs
holds, for $\sigma(t)=\omega\left(t^{\frac{r}{mh}}\right)$, then $Q$ is
$h$-weaker than $P$.
\end{Th}

\begin{proof}
\underline{Roumieu case:}

We follow the same ideas of Juan-Huguet in \cite{J}, 
substituting to the assumption, in \cite[Thm. 4.5]{J}, that
the single operator $P(D)$ is hypoelliptic, the weaker
assumption that the system
$P$ satisfies condition $(\condC)$, in the spirit of \cite{BC}.

Let us then assume \eqref{R2} to be satisfied and fix a compact set 
$K_0\subset\Omega$.

We have the following inclusions:
\beqsn
\E^P_{(\omega)}(\Omega)\subseteq\E^P_{\{\omega\}}(\Omega)
\subseteq\E^Q_{\{\sigma\}}(\Omega)
=\proj_{\afrac{\longleftarrow}{\K\subset\subset\Omega}}
\ind_{\afrac{\longrightarrow}{\ell\in\N}}\E^Q_{\sigma,\frac1\ell}(K)
\subseteq
\ind_{\afrac{\longrightarrow}{\ell\in\N}}\E^Q_{\sigma,\frac1\ell}(K_0).
\eeqsn

By assumption the system $P$ satisfies condition $(\condC)$ and hence, 
by Remark~\ref{rem210bis},
$\E^P_{(\omega)}(\Omega)$ is a 
Fr\'echet space and
$\ds\ind_{\afrac{\longrightarrow}{\ell\in\N}}\E^Q_{\sigma,\frac1\ell}(K_0)$ is an 
(LF)-space.
We can therefore apply the Closed Graph Theorem and Grothendieck's
Factorization Theorem (see \cite[Thms 24.31 and 24.33]{MV}) and obtain that
there exists $\ell_0\in\N$ such that
\beqsn
\E^P_{(\omega)}(\Omega)\subseteq\E^Q_{\sigma,\frac{1}{\ell_0}}(K_0)
\eeqsn
with a continuous inclusion.

There exist then a constant $C>0$, a compact
$K\subset\subset\Omega$ and $\lambda>0$ such that, 
for all $f\in\E^P_{(\omega)}(\Omega)$:
\beqs
\label{40}
\sup_{\beta\in\N_0^M}\|Q^\beta(D)f\|_{L^2(K_0)}
e^{-\frac{1}{\ell_0}\varphi^*_\sigma(|\beta|r\ell_0)}
\leq C\sup_{\alpha\in\N_0^N}\|P^\alpha(D)f\|_{L^2(K)}
e^{-\lambda\varphi^*_\omega\left(\frac{|\alpha|m}{\lambda}\right)}.
\eeqs

For $\xi\in\R^n$, we denote $f_\xi(x):=e^{i\langle x,\xi\rangle}$ and remark
that
 $f_\xi\in\E^P_{(\omega)}(\Omega)$, because for every compact 
$K\subset\subset\Omega$ and $\lambda>0$
\beqsn
\|P^\alpha(D)f_\xi\|_{L^2(K)}=\|P^\alpha(\xi)f_\xi\|_{L^2(K)}
\leq m(K)|P^\alpha(\xi)|\leq C(1+|\xi|^{m|\alpha|})
\leq C_\xi e^{\lambda'\varphi^*_\omega\left(\frac{|\alpha|m}{\lambda'}\right)}
\eeqsn
for some $C_\xi>0$ and $\lambda'>0$, by \eqref{la5}.
Since $f_\xi\in\E^P_{(\omega)}(\Omega)$ and we can
apply \eqref{40} to $f_\xi$, obtaining that
\beqsn
\sup_{\beta\in\N_0^M}|Q^\beta(\xi)|e^{-\frac{1}{\ell_0}\varphi^*_\sigma(|\beta|r\ell_0)}
\leq C'\sup_{\alpha\in\N_0^N}|P^\alpha(\xi)|
e^{-\lambda\varphi^*_\omega\left(\frac{|\alpha|m}{\lambda}\right)}
\eeqsn
for some $C'>0$.

Therefore
\beqs
\nonumber
&&\sup_{\beta\in\N_0^M}
\bigg(\sum_{j=1}^M\left|\frac{Q_j}{2}(\xi)\right|\bigg)^{|\beta|}
e^{-\frac{1}{\ell_0}\varphi^*_\sigma(|\beta|r\ell_0)}\\
\nonumber
\leq&&\sup_{\beta\in\N_0^M}
\left(\sum_{\beta_1+\ldots+\beta_N=|\beta|}\frac{|\beta|!}{\beta_1!\cdots\beta_N!}
|Q_1(\xi)|^{\beta_1}\cdots|Q_M(\xi)|^{\beta_M}\frac{1}{2^{|\beta|}}
e^{-\frac{1}{\ell_0}\varphi^*_\sigma(|\beta|r\ell_0)}\right)\\
\nonumber
\leq&&
\sup_{\beta\in\N_0^M}|Q^\beta(\xi)|e^{-\frac{1}{\ell_0}\varphi^*_\sigma(|\beta|r\ell_0)}
\leq C'\sup_{\alpha\in\N_0^N}|P^\alpha(\xi)|
e^{-\lambda\varphi^*_\omega\left(\frac{|\alpha|m}{\lambda}\right)}\\
\label{44}
\leq&&C'' \sup_{\alpha\in\N_0^N}\bigg(\sum_{j=1}^N|P_j(\xi)|\bigg)^{|\alpha|}
e^{-\lambda\varphi^*_\omega\left(\frac{|\alpha|m}{\lambda}\right)}.
\eeqs

From Lemma~\ref{lemmaJ} it follows that, if
$\sum_{j=1}^M|\frac{Q_j}{2}(\xi)|\geq1$ and $\sum_{j=1}^N|P_j(\xi)|\geq1$, then 
\beqs
\label{41}
\hspace*{7mm}
\bigg(\sum_{j=1}^M\bigg|\frac{Q_j}{2}(\xi)\bigg|\bigg)^{-1}\!\!
\exp\left\{\frac{1}{\ell_0}\sigma\left(\bigg(\sum_{j=1}^M\bigg|\frac{Q_j}{2}(\xi)\bigg|
\bigg)^\frac1r\right)\right\}
\leq\tilde{C}\exp\left\{\lambda\omega\left(\bigg(\sum_{j=1}^N|P_j(\xi)|
\bigg)^{\frac1m}\right)\right\},
\eeqs
for some $\tilde{C}>0$.

From property $(\gamma)$ of the weight function $\sigma(t)$ we have that
\eqref{41} implies, for some $\lambda'>0$, if
$\sum_{j=1}^M|\frac{Q_j}{2}(\xi)|\geq1$ and $\sum_{j=1}^N|P_j(\xi)|\geq1$:
\beqsn
\exp\left\{\lambda'\sigma\left(\bigg(\sum_{j=1}^M\bigg|\frac{Q_j}{2}(\xi)\bigg|\bigg)^\frac1r\right)
\right\}\leq\tilde{C}
\exp\left\{\lambda\omega\left(\bigg(\sum_{j=1}^N|P_j(\xi)|
\bigg)^\frac1m\right)\right\}.
\eeqsn 

Since $\sigma(t)=\omega\left(t^{\frac{r}{mh}}\right)$ by assumption,
we thus obtain:
\beqs
\nonumber
\omega\left(\bigg(\sum_{j=1}^M\bigg|\frac{Q_j}{2}(\xi)\bigg|\bigg)^{\frac{1}{mh}}\right)\leq&& A
\left(1+\omega\left(\bigg(\sum_{j=1}^N|P_j(\xi)|
\bigg)^{\frac1m}\right)\right)\\
\label{42}
\leq&&\omega\left(A'\bigg(\sum_{j=1}^N|P_j(\xi)|
\bigg)^{\frac1m}\right)
\eeqs
for some $A'>0$ if $\sum_{j=1}^N|P_j(\xi)|\geq1$ and
$\sum_{j=1}^M|\frac{Q_j}{2}(\xi)|\geq1$, because condition B-M-M implies that
for every $k\in\N$ there exists a constant $H_k\geq1$ such that 
$2^{k-1}\omega(t)\leq \omega(H_kt)$ for all $t\geq1$.

Since $\omega(t)$ is increasing, \eqref{42} implies that there exists a
constant $B>1$ such that
\beqs
\label{43}
\sum_{j=1}^M|Q_j(\xi)|\leq B\left(1+\sum_{j=1}^N|P_j(\xi)|\right)^h,
\eeqs
if $\sum_{j=1}^M|\frac{Q_j}{2}(\xi)|\geq1$ and $\sum_{j=1}^N|P_j(\xi)|\geq1$.

However, \eqref{43} is trivial if $\sum_{j=1}^M|\frac{Q_j}{2}(\xi)|\leq1$ or
$\sum_{j=1}^N|P_j(\xi)|\leq1$, so that  \eqref{43} is satisfied for all 
$\xi\in\R^n$ and $Q$ is $h$-weaker than $P$.

\underline{Beurling case:}

The proof is similar, but easier, as in the Roumieu case, since
$\E^P_{(\omega)}(\Omega)$ and $\E^Q_{(\sigma)}(\Omega)$ are metrizable,
and hence the inclusion \eqref{B2} implies
\eqref{40}.
\end{proof}

\begin{Rem}
\begin{em}
By Remark~\ref{rem210tris}, 
instead of condition $(\condC)$ we can consider, in Theorem~\ref{th2}, the
weaker assumption that $\E_{(\omega)}^P(\Omega)$ is a Fr\'echet space and then
take on $\E^Q_{\sigma,1/\ell}(K_0)$ the fundamental system of semi-norms
$\{\tau^Q_{K_0,\ell,m}\}_{m\in\N}$ defined by \eqref{defsemi}, to make
$\ds\ind_{\afrac{\longrightarrow}{\ell\in\N}}\E^Q_{\sigma,1/\ell}(K_0)$
an (LF)-space.
\end{em}
\end{Rem}

As a consequence of Theorem~\ref{th2} we have the converse of
Corollary~\ref{cor1}:
\begin{Cor}
\label{cor2}
Let $\Omega$ be an open subset of $\R^n$.
Let $\omega$ be a non-quasianalytic weight function satisfying condition
B-M-M, and let $P=(P_j(D))_{j=1}^N$ be a system of order $m$ satisfying
condition $(\condC)$.
If
\beqs
\label{45}
\E^P_{(\omega)}(\Omega)\subseteq\E_{(\omega)}(\Omega),
\eeqs
or
\beqs
\label{45bis}
\E^P_{\{\omega\}}(\Omega)\subseteq\E_{\{\omega\}}(\Omega),
\eeqs
then the system $P$ is elliptic.
\end{Cor}

\begin{proof}
\underline{Beurling case:}

Let us consider the system $Q=(D_j)_{j=1}^n$.
Then $\E^Q_{(\omega)}(\Omega)=\E_{(\omega)}(\Omega)$ and \eqref{45} implies
\eqref{B2} with $\sigma(t)=\omega(t)=\omega\left(t^{\frac{r}{mh}}\right)$
for $r=1$ and $h=1/m$.

By Theorem~\ref{th2} we have that $Q$ is $\frac1m$-weaker than $P$, i.e.
\beqsn
\sum_{j=1}^n|\xi_j|\leq C\bigg(1+\sum_{j=1}^N|P_j(\xi)|\bigg)^{\frac1m},
\qquad\forall\xi\in\R^n.
\eeqsn

This proves that the system $P$ is elliptic, and hence the corollary is
proved.

\underline{Roumieu case:}

The proof is similar as in the Beurling case, using \eqref{45bis} and
\eqref{R2} instead of \eqref{45} and \eqref{B2}.
\end{proof}

\end{document}